\documentclass[12pt,english]{amsart}
\usepackage[english]{babel}
\usepackage{amsmath,amsthm,amssymb}
\usepackage{xcolor}
\usepackage[colorlinks]{hyperref}
\usepackage{comment}

\newtheorem{theorem}{Theorem}[section]
\newtheorem{lemma}[theorem]{Lemma}
\newtheorem{corollary}[theorem]{Corollary}
\newtheorem{remark}[theorem]{Remark}

\newtheorem{proposition}[theorem]{Proposition}
\theoremstyle{definition}

\numberwithin{equation}{section}

\begin{document}

\title[]
{Reverse Urysohn inequality for functions and functional Shephard problem}

\author{V. Gorev}
\thanks{This research was supported by the Russian Science Foundation Grant 25-11-00007 at Lomonosov Moscow
State University.}

\maketitle

\begin{abstract}
The results obtained in this paper provide a generalization of some geometric inequalities to the class of log-concave functions. They describe the relationships between the Quermassintegrals of functions and bodies, giving estimates for constants involved in functional Shephard-type problems.
\end{abstract}


\section{Introduction}
One geometrical consequence of Fubini's theorem is the following fact:

{\it
Let $K$ and $V$ be two bodies in $\mathbb{R}^n$ with $n \ge 2$, $H_0$ be a fixed $(n-1)$-dimensional hyperplane. Assume that, for any $(n-1)$-dimensional hyperplane $H$ parallel to $H_0$, the inequality
$$
\lambda_{n-1}(S_H(K)) \le \lambda_{n-1}(S_H(V))
$$
holds, where $S_H(M)$ is the intersection of the set $M$ with the hyperplane $H$. Then we have
$$
\lambda_n(K) \le \lambda_n(V).
$$}

Thus, inequalities for the volumes of $(n-1)$-dimensional sections of bodies imply inequalities for the volumes of the bodies themselves. There is a natural question are there any other configurations of sections with the same property.
Naturally, the question arises: are there any other configurations of sections with the same property?

%Отметим, что из-за необходимости понижения размерности все приведенные задачи рассматриваются только в пространствах $\mathbb{R}^n$ для $n \ge 2$.

Instead of comparing volumes of sections in parallel planes, one can compare volumes in planes passing through a fixed point (for example, $0$). The corresponding problem is called the Busemann-Petty problem.

{\bf The classical Busemann-Petty problem.}
Let $K$ and $V$ be two centrally symmetric convex bodies in $\mathbb{R}^n$, and let the inequality
$$
\lambda_{n-1}(S_H(K)) \le \lambda_{n-1}(S_H(V))
$$
holds for every $(n-1)$-dimensional hyperplane $H$ passing through the point $0$. Is it true that
$$
\lambda_n(K) \le \lambda_n(V)?
$$
Here $S_H(M)$ denotes the intersection of the set $M$ with the hyperplane $H$.

It turns out that the answer to this question is positive for the dimensions $n \le 4$ and negative for $n \ge 5$ (see \cite[Theorem 5.4]{Kold}).

In a certain sense, the dual concept to a section of a body is its projection.
Thus, we arrive at the the following problem.

{\bf The classical Shephard problem.}
Let $K$ and $V$ be two centrally symmetric convex bodies in $\mathbb{R}^n$, and let the inequality
$$
\lambda_{n-1}(P_H(K)) \le \lambda_{n-1}(P_H(V))
$$
holds for any $(n-1)$-dimensional hyperplane $H$ passing through the point $0$. Is it true that
$$
\lambda_n(K) \le \lambda_n(V)?
$$
Here $P_H(M)$ denotes the orthogonal projection of the set $M$ onto the hyperplane $H$.

It turns out that the answer to this question is generally negative. Namely, it is positive only for the dimension $n = 2$ (see \cite[Theorem 8.12]{Kold}). 

Thus, in higher dimensions, the corresponding inequalities are not valid. In this regard, the problem of finding the asymptotics (or at least the most accurate upper estimate) arises for the following two sequences:
$$
S^S_{n,k} = \inf \left\{ \frac{\lambda_n(V)}{\lambda_n(K)} \colon \text{$K$, $V$ are convex bodies such that} \forall H \in Gr_{n,k} \colon \lambda_{k}(P_H(K)) \le \lambda_{k}(P_H(V)) \right\}
$$
and
$$
S^{BP}_{n,k} = \inf \left\{ \frac{\lambda_n(V)}{\lambda_n(K)} \colon \text{$K$, $V$ are convex bodies such that} \forall H \in Gr_{n,k} \colon \lambda_{k}(S_H(K)) \le \lambda_{k}(S_H(V)) \right\},
$$
where $Gr_{n,k}$ is the set of all $k$-dimensional hyperplanes in $\mathbb{R}^n$ passing through the point $0$.

Many proofs of geometric inequalities involve various integral transformations applied to the indicator functions of corresponding bodies. Therefore, a natural step is to study possible functional generalizations of such geometric problems (see, for example,\cite{A-GMJR}, \cite{A-G}, \cite{A-AKM}, \cite{BCF}, \cite{KM}).
This work continues the research initiated in \cite{GoK}, where functional Busemann--Petty and Shephard problems were considered for the class of logarithmically concave functions. In particular, this work provides a generalization of the following result from \cite{GK} (see Theorem 4.8 there).

\begin{theorem}
\label{th4}
Let $K$, $V$ be convex bodies in $\mathbb{R}^n$, and let the Shephard condition
$$
|P_H(K)| \le |P_H(V)| \quad \forall H \in Gr_{n,k}
$$
hold. Then
$$
|K|^{\frac{1}{n}} \le c \ln(n) |V|^{\frac{1}{n}},
$$
where $c>0$ is an absolute constant.
\end{theorem}

One of the important steps in proving this theorem is the following statement (see \cite[Theorem 6.5.4]{A-AGM} and \cite[Theorem 1.11.5]{BGVV}).

\begin{theorem}[Reverse Urysohn inequality]
\label{prop1}
Let $A$ be a convex body. Then
$$\min\{W_{n-1}(TA) \colon T \in GL(n),\, \lambda_n(TA) = 1\}\le C \omega_n \sqrt{n} \ln(n),$$
where $W_{n-1}(A)$ is the $(n-1)$-th Quermassintegral of the body $A$, $\omega_n$ is the volume of the Euclidean unit ball of radius 1, and $C>0$ is an absolute constant.
In other words,
$$
\max \Big\{\min_{T \in GL(n),\, \lambda_n(TA) = 1} W_{n-1}(TA) \colon A\text{ is a convex body} \Big\} \le C \omega_n \sqrt{n} \ln(n).
$$
\end{theorem}

In this work, Theorems \ref{th4} and \ref{prop1} are generalized to the class of logarithmically concave functions. Specifically, the following results are obtained (detailed notation is given in the next section).

\begin{theorem}
\label{th0}
Let $f_1, f_2 \in Q^n$ and satisfy the Shephard condition:
$$
\|P_H f_1\|_1 \le \|P_H f_2\|_1 \quad \forall H \in Gr_{n,k}.
$$
Then the inequality
$$
\|f_1\|_{\frac{n}{k}} \le (C \ln(n))^k \|f_2\|_{1}^{\frac{k}{n}} \|f_2\|_{\infty}^{\frac{n-k}{n}},
$$
holds. Equivalently,
$$
\left(\int_{\mathbb{R}^n} f_1^{\frac{n}{k}} \, dx\right)^{\frac{1}{n}} \le C \ln(n) \|f_2\|_{\infty}^{\frac{n-k}{kn}}
\left(\int_{\mathbb{R}^n} f_2 \, dx\right)^{\frac{1}{n}},
$$
where $C > 0$ is an absolute constant.
\end{theorem}

\begin{theorem}
\label{th6}
Let $f \in Q^n$, then
$$
\min\{W_{n-1} (T f) \colon T \in GL(n),\, \|T f\|_1 = 1 \} \le C \|f\|_{\infty}^{\frac{n-1}{n}} \omega_n \sqrt{n} \ln(n),
$$
where $C > 0$ is an absolute constant.
\end{theorem}

Note that a stronger statement, that connects Quermassintegrals of logarithmically concave functions with Quermassintegrals of their level sets (see Theorem \ref{prop2} below), is proved in this work.

Section $2$ of this work contains necessary definitions and notation. In Section $3$, we prove the inequalities that connect Quermassintegrals of functions with Quermassintegrals of bodies constructed from these functions. In Section $4$, these inequalities are applied to the proof of Theorem \ref{th0}.

\section{Definitions and notation}

Let $\langle\cdot, \cdot\rangle$ be the standard inner product on $\mathbb{R}^n$, and $|\cdot|$ be the standard Euclidean norm on $\mathbb{R}^n$, i.e., $|x|:=\sqrt{\langle x, x\rangle}$. Let $\lambda_n$ be the standard Lebesgue measure on $\mathbb{R}^n$.
The volume of the unit Euclidean ball
$B_2^n:=\{x\colon |x|\le 1\}$ in $\mathbb{R}^n$
is denoted by $\omega_n$,
i.e. $\omega_n:=\lambda_n(x\colon |x|\le 1)$.
For $p \in (-\infty, 0) \cup (0, +\infty)$, let $L^p(\mathbb{R}^n)$ denote the space of functions on $\mathbb{R}^n$ integrable to the power $p$.
On the space $L^p(\mathbb{R}^n)$, we define the functional
$$
\|f\|_{L^p(\mathbb{R}^n)}=\|f\|_p:=\Bigl(\int{_\mathbb{R}^n}|f(x)|^p dx\Bigr)^{\frac{1}{p}},
$$
which is a norm on this space for $p \ge 1$.
Furthermore, for a bounded function $f$ on $\mathbb{R}^n$, we define $\|f\|_\infty:=\sup\limits_{x\in\mathbb{R}^n}|f(x)|$.

A function $f\colon\mathbb{R}^n\to \mathbb{R}$ is said to be upper semi-continuous if the level sets 
$$
A_t(f):=\{x\in \mathbb{R}^n\colon f(x)\ge t\}
$$ 
are closed for every $t\in\mathbb{R}$.

A function $f\colon \mathbb{R}^n\to \mathbb{R}_+ := [0, +\infty)$ is called logarithmically concave (log-concave) if
$$
f(sx+(1-s)y)\ge [f(x)]^s\cdot[f(y)]^{1-s}\quad \forall x, y\in \mathbb{R}^n, \forall s\in(0, 1).
$$
In particular, the sets $A_t(f) = \{x\in \mathbb{R}^n\colon f(x)\ge t\}$ are convex for any $t$ and any logarithmically concave function $f$.

The main object studied in this paper is the class of logarithmically concave functions, namely, the class
$$
Q^n := \Big\{
f\colon \mathbb{R}^n\to \mathbb{R}_+ \colon f \hbox{ is log-concave and upper semicontinuous},
0 <\int_{\mathbb{R}^n} f \, dx < +\infty 
\Big\}.
$$

Recall that a probability Borel measure $\mu$ in $\mathbb{R}^n$ is called logarithmically concave if
$$
\mu(\lambda A + (1 - t) B) \ge \mu(A)^{t} \mu(B)^{1 - t}
$$
for every pair of Borel sets $A$, $B$ and $t \in [0, 1]$. Note that every probability density $f$ from $Q^n$ is the density of some logarithmically concave measure (see \cite[Theorem 4.3.2]{Bog}, \cite[Theorem 1.1]{Bor}).

Let $Gr_{n,k}$ be the Grassmannian manifold consisting of all possible $k$-dimensional planes in $\mathbb{R}^n$, and let $\nu_{n,k}$ be the standard Haar probability measure on it.
Let $A \subseteq \mathbb{R}^n$ be a convex body, then its $(n-k)$-th Quermassintegral is defined as
$$
W_{n-k}(A) = \frac{\omega_n}{\omega_k} \int_{Gr_{n,k}} \lambda_k(P_H(A)) \, \nu_{n,k} (dH),
$$
where $P_H(A)$ is the orthogonal projection of the body $A$ onto the plane $H$, and $\omega_n$ is the volume of the $n$-dimensional Euclidean ball of radius $1$.

Let $f \in Q^n$, then its $(n-k)$-th Quermassintegral is defined as the quantity (see \cite[Definition~3.1]{BCF})
\begin{equation}
\label{eq0}
W_{n-k} (f) := \int_0^{\|f\|_{\infty}} W_{n-k} (A_t(f)) \, dt.
\end{equation}
By Fubini's theorem, one can readily verify
$$
W_{n-k} (f) = \frac{\omega_n}{\omega_k}
\int_{Gr_{n,k}} \|P_H f\|_{L^1(H)}\, \nu_{n, k}(dH),
$$
where $P_H(f) \colon H \rightarrow \mathbb{R}_+$ is the projection of the function $f$ onto the plane $H$ defined by the formula
$$
P_H f(y) := \sup\limits_{z\in H^\bot}f(y+z),
$$

To simplify the formulas, we introduce the following notation for level sets:
$$
A_{[u]}(f) := A_{\|f\|_{\infty} e^{-u}} (f).
$$

We will need the following function $g_k(f)\colon \mathbb{R}_+ \rightarrow \mathbb{R}_+$,
\begin{equation}
\label{gk}
g_k(f)(u) := W_{n-k}^{\frac{1}{k}} (A_{[u]} (f)),
\end{equation}
which has the important property of concavity (see Lemma \ref{lemma3} below).
It is worth noting that on indicator functions of convex bodies, the function $g_k$ is constant, i.e. $g_k(I_M) = W_{n-k}^{\frac{1}{k}} (M)$.

Making the change of variables
$$
t = \|f\|_{\infty} e^{-u} \Longleftrightarrow u = -\ln\left( 
\frac{t}{\|f\|_{\infty}} \right)
$$
in the integral \eqref{eq0} we get $W_{n-k}(A_t(f)) = g_k(f)^k(u)$ and 
\begin{equation}
\label{eq1}
W_{n-k}(f) = \int_0^{\|f\|_{\infty}} g_k(f)^k\left(-\ln\left( 
\frac{t}{\|f\|_{\infty}} \right)\right)\, dt = 
\|f\|_{\infty} \int_0^{\infty} g_k(f)^k(u) e^{-u} \, du.
\end{equation}

For $T\in GL(n)$ and $f\in Q^n$, let $Tf\colon\mathbb{R}^n\rightarrow\mathbb{R}_+$ denote the function
$$
Tf(x) := f(T^{-1} x).
$$
Note that, for indicator functions of sets, $f = I_A$, the equality $T I_A = I_{T A}$ holds.

We need the following inequality, called Alexandrov's inequality (see \cite[section 1.4.4]{BGVV}) and \cite[inequality(7.67)]{Sch}): Let $K$ be a convex body in $\mathbb{R}^n$, then
\begin{equation}
\label{alex}
\left( \frac{W_i(K)}{\omega_n} \right)^{\frac{1}{n-i}} \ge
\left( \frac{W_j(K)}{\omega_n} \right)^{\frac{1}{n-j}},
\quad n > i > j \ge 0
\end{equation}

\section{Relationship between Quermassintegrals of functions and their level sets}

The main result of this section is the following theorem.

\begin{theorem}
\label{prop2}
Let $k \in \{1, \ldots, n\}$ and $f \in Q^n$, then
$$\left( \frac{k}{e \, n} \right)^k 
\le \frac{W_{n-k}(f)}{\|f\|_1^{\frac{k}{n}} \|f\|_{\infty}^{\frac{n-k}{n}} W_{n-k} \left(
\frac{A_{[k]}(f)}{\lambda_n(A_{[k]}(f))^{\frac{1}{n}}}\right)} \le e^{\frac{k^2}{n}} \le e^k.$$
\end{theorem}

We need several auxiliary statements to prove this theorem.

\begin{lemma}
\label{lemma3}
The function $g_k(f)$ defined by the equation \eqref{gk} is: 1) non-negative, 2) non-decreasing, and 3) concave.
\end{lemma}

\begin{proof}
Recall that the function $g_k(f)$ has the following form:
$$
g_k(f)(u) = W_{n-k}^{\frac{1}{k}} (A_{\|f\|_{\infty} e^{-u}}(f)).
$$

Property $1)$ follows from the non-negativity of the Quermassintegrals $W_{n-k}$. 
Property $2)$ follows from the monotonicity of Quermassintegrals and the monotonicity of level sets.

Let's prove property $3)$. Let $f(x) \ge t \ge 0$, $f(y) \ge s \ge 0$ and $\alpha \in [0, 1], \beta = 1 - \alpha$, then
$$f(\alpha x + \beta y) \ge f(x)^\alpha f(y)^\beta \ge t^\alpha s^\beta.$$
Therefore,
$$\alpha \{f \ge t\} + \beta \{f \ge s\} \subseteq \{ f \ge t^\alpha s^\beta \},$$
and, equivalently,
$$\alpha A_t(f) + \beta A_s(f) \subseteq A_{t^\alpha s^\beta}(f).$$

We now use the Brunn--Minkowski inequality for Quermassintegrals (see \cite[Theorem 7.4.5]{Sch} and formula (5.31) there):
\begin{equation}
\label{eq2}
W_{n-k}(A+B)^{\frac{1}{k}} \ge W_{n-k}(A)^{\frac{1}{k}} + W_{n-k}(B)^{\frac{1}{k}},
\end{equation}
where $A$, $B$ are convex bodies in $\mathbb{R}^n$ and $k \in \{1, \ldots, n\}$. From the monotonicity of the Quermassintegral $W_{n-k}$ and inequality \eqref{eq2}, it follows that
$$W_{n-k}^{\frac{1}{k}}(A_{t^\alpha s^\beta}(f)) \ge W_{n-k}^{\frac{1}{k}} (\alpha A_t(f) + \beta A_s(f)) \ge \alpha W_{n-k}^{\frac{1}{k}}(A_t(f)) + \beta W_{n-k}^{\frac{1}{k}} (A_s(f)),$$
Substituting $t = \|f\|_{\infty} e^{-u}$ and $s = \|f\|_{\infty} e^{-v}$ in the inequality above, where $u, v \in [0, +\infty]$, we obtain the concavity of the function $g_k(f)$, and the lemma is proved.
\end{proof}

\begin{lemma}
Let $g \colon \mathbb{R}_+ \rightarrow \mathbb{R}_+$ be a non-negative, non-decreasing, concave, continuous at $0$ function. Then it can be uniformly approximated with any accuracy by a $C^1$-smooth non-negative, non-decreasing, concave function $\tilde{g} \colon \mathbb{R}_+ \rightarrow \mathbb{R}_+$.
\end{lemma}
\begin{proof}
We firstly show that the function $g$ can be approximated by piecewise linear functions. 
The concavity of function $g$ implies its continuity on $\mathbb{R}_+$
Let's find $\delta > 0$ such that $g(\delta) - g(0) \le \varepsilon$.
From the non-decreasing condition, it follows that
$$g(0) \le g(u) \le g(0) + \varepsilon \quad \forall u \in [0, \delta].$$

From the concavity condition, we obtain that $g(2 \delta) \le g(\delta) + \frac{g(\delta) - g(0)}{\delta} \delta \le g(\delta) + \varepsilon$. Therefore,
$$g(\delta) \le g(u) \le g(\delta) + \varepsilon \quad \forall u \in [\delta, 2 \delta].$$
Similarly
$$g(N\delta) \le g(u) \le g(N\delta) + \varepsilon \quad \forall N \in \mathbb{N}\cup\{0\}, \, u \in [N\delta, (N+1)\delta].$$

The same condition holds for the piecewise linear function $\hat{g}$ with vertices at $(N\delta, g(N\delta))$, $N \in \mathbb{N}\cup{0}$, i.e.
$$\hat g(N\delta) \le \hat g(u) \le \hat g(N\delta) + \varepsilon \quad \forall N \in \mathbb{N}\cup\{0\} \, \forall u \in [N\delta, (N+1)\delta].$$
Therefore, the piecewise linear function $\hat g$ uniformly approximates the function $g$ with accuracy $\varepsilon$ on the entire set $\mathbb{R}_+$.

Now, using the piecewise linear function $\hat g$, we construct the $C^1$-smooth function $\tilde{g}$ as follows: we replace the segments of the graph of $\hat g$ containing corners with circular arcs in such a way that the resulting function differs from the original one by no more than $\varepsilon$ in the $\| \cdot \|_{\infty}$ norm.
Then, $\tilde{g}$ will differ from $g$ by no more than $2 \varepsilon$. The lemma is proved.
\end{proof}

\begin{lemma}
\label{lemma5}
Let $m \ge 0, q > 0$, and let $g\colon \mathbb{R}_+ \rightarrow \mathbb{R}_+$ be a non-negative, non-decreasing, concave function. Then
$$
q \int_0^{+\infty} g^m(u) u^{q-1} e^{-u} \, du \le
\int_0^{+\infty} g^m(u) u^q e^{-u} \, du \le
(m+q) \int_0^{+\infty} g^m(u) u^{q-1} e^{-u} \, du.
$$
In particular, for $g \not \equiv 0$,
$$q \le \frac{\int_0^{+\infty} g^m(u) u^q e^{-u} \, du}{\int_0^{+\infty} g^m(u) u^{q-1} e^{-u} \, du} \le m+q.$$
\end{lemma}

\begin{proof}
Without loss of generality we consider the case of a smooth function $g$.
Otherwise, we approximate it with smooth functions by previous lemma.
Integrating by parts, we get
\begin{align*}
\int_0^{+\infty} g^m(u) u^q e^{-u} \, du &= -\int_0^{+\infty} g^m(u) u^q \, d(e^{-u}) = \int_0^{+\infty} e^{-u} \, d(u^q g^m(u)) =\\
&= \int_0^{+\infty} e^{-u} g^m(u) q u^{q-1} \, du + \int_0^{+\infty} e^{-u} m g^{m-1}(u) g'(u) u^q \, du.
\end{align*}

Now the lower bound follows from the non-negativity of the second term above. To prove the upper bound, we use the inequality $g'(x) \le \frac{g(x) - g(0)}{x} \le \frac{g(x)}{x}$ for $x>0$ in the second term above and obtain
\begin{align*}
&\int_0^{+\infty} e^{-u} g^m(u) q u^{q-1} \, du + \int_0^{+\infty} e^{-u} m g^{m-1}(u) g'(u) u^q \, du \le\\
&\le \int_0^{+\infty} e^{-u} g^m(u) q u^{q-1} \, du + \int_0^{+\infty} e^{-u} m g^m(u) u^{q-1} \, du =\\
&= (m+q) \int_0^{+\infty} g^m(u) u^{q-1} e^{-u} \, du.
\end{align*}
The lemma is proved.
\end{proof}

\begin{corollary}
\label{cor0}
Let $\forall m \ge 0, q > 0$, and let $g\colon \mathbb{R}_+ \rightarrow \mathbb{R}_+$ be a non-negative, non-decreasing, concave function.
Then the following inequality holds
$$
\int_0^{+\infty} g^{m+1}(u) u^{q-1} e^{-u} \, du \le \frac{g(q)}{q} \int_0^{+\infty} g^{m}(u) u^{q} e^{-u} \, du.
$$
\end{corollary}
\begin{proof}
If $g \equiv 0$ everywhere on $\mathbb{R}_+$, the statement is trivial. Assume now that $g \not \equiv 0$.

First, using Jensen's inequality for the measure with the density $g^m(u) u^{q-1} e^{-u} \, du$, we get
$$
\int_0^{+\infty} g^{m+1}(u) u^{q-1} e^{-u} \, du \le
g\left( \frac{\int_0^{+\infty} g^m(u) u^q e^{-u} \, du}{\int_0^{+\infty} g^m(u) u^{q-1} e^{-u} \, du} \right) \int_0^{+\infty} g^m(u) u^{q-1} e^{-u} \, du.
$$
By the lower estimate from Lemma \ref{lemma5} we get $x \le y$, where
$$
x = q, \quad y = \frac{\int_0^{+\infty} g^m(u) u^q e^{-u} \, du}{\int_0^{+\infty} g^m(u) u^{q-1} e^{-u} \, du}.
$$
Now the inequality $g(y) \le g(x) \frac{y}{x}$ for $0<x\le y$ yields
$$
g\left( \frac{\int_0^{+\infty} g^m(u) u^q e^{-u} \, du}{\int_0^{+\infty} g^m(u) u^{q-1} e^{-u} \, du} \right) \le \frac{g(q)}{q}
\frac{\int_0^{+\infty} g^m(u) u^q e^{-u} \, du}{\int_0^{+\infty} g^m(u) u^{q-1} e^{-u} \, du},
$$
implying the statement of the corollary.
\end{proof}

\begin{corollary}
\label{cor1}
Let $k \in \mathbb{N}, p \ge 0$, and let $g\colon \mathbb{R}_+ \rightarrow \mathbb{R}_+$ be a non-negative, non-decreasing, concave function.
Then
$$g^k(p) e^{-p} \le \int_0^{+\infty} g^k(u) e^{-u} \, du \le g(1) \cdot g(2) \cdot ... \cdot g(k) \le g^k(k).$$
\end{corollary}

\begin{proof}
For the lower bound, we only need non-negativity and non-decreasing properties. Indeed,
$$\int_0^{+\infty} g^k(u) e^{-u} \, du \ge \int_{p}^{+\infty} g^k(u) e^{-u} \, du \ge g^k(p) e^{-p}.$$

Now let's prove the upper bound. Applying Corollary \ref{cor0} inductively, we obtain
\begin{multline*}
\int_0^{+\infty} g^k(u) e^{-u} \, du \le g(1) \int_0^{+\infty} g^{k-1}(u) u e^{-u} \, du.
\le g(1) \frac{g(2)}{2} \int_0^{+\infty} g^{k-2}(u) u^2 e^{-u} \, du \le
... \le\\
\le g(1) \frac{g(2)}{2} ... \frac{g(k)}{k} \int_0^{+\infty} u^k e^{-u} \, du = g(1) g(2) ... g(k).
\end{multline*}
The lemma is proved.
\end{proof}

\begin{proposition}
\label{cor2}
Let $k \in \{1, \ldots, n\}$ and $f \in Q^n$, then
$$\frac{1}{e} W_{n-k}^{\frac{1}{k}}(A_{[k]}(f)) \le \left(\frac{1}{\|f\|_{\infty}} W_{n-k}(f)\right)^{\frac{1}{k}} \le W_{n-k}^{\frac{1}{k}}(A_{[k]}(f)).$$
\end{proposition}
\begin{proof}
Recall that by the formula \eqref{eq1},
$$
\frac{1}{\|f\|_{\infty}} W_{n-k}(f) = \int_0^{\infty} g_k(f)^k(u) e^{-u} \, du.
$$
Using Corollary \ref{cor1} with the function $g = g_k(f)$ and $p=k$, we get
$$\frac{1}{\|f\|_{\infty}} W_{n-k}(f) = \int_0^{\infty} g_k(f)^k(u) e^{-u} \, du \le g_k(f)^k(k) = W_{n-k}(A_{[k]}(f))$$
where the last equality is the definition of the function $g_k$ from \eqref{gk}.
By the same corollary we have
$$\frac{1}{\|f\|_{\infty}} W_{n-k}(f) \ge g_k(f)^k(k) e^{-k} = W_{n-k}(A_{[k]}(f)) e^{-k}.$$
The proposition is proved.
\end{proof}

\begin{lemma}
\label{lemma1}
Let $k \in \{1, \ldots, n\}$ and $f \in Q^n$, then
$$\left( \frac{k}{n} \right)^{n} \frac{1}{\|f\|_{\infty}}\|f\|_1 \le \lambda_n(A_{[k]}(f)) \le e^k \frac{1}{\|f\|_{\infty}} \|f\|_1$$
\end{lemma}

\begin{proof}
The upper bound follows from Chebyshev's inequality. Let's explain the lower bound.
By Corollary \ref{cor1} we get
$$\frac{1}{\|f\|_{\infty}} \|f\|_1 = \frac{1}{\|f\|_{\infty}} W_{n-n}(f) = \int_0^{+\infty} g_n(f)^n(u) e^{-u} \, du \le g_n(f)^n(n).$$

Furthermore, from the concavity of the function $g_n(f)$ (see Lemma \ref{lemma3}), we obtain
$$g_n(f)(n) \le \frac{n}{k} g_n(f)(k) = \frac{n}{k} \lambda_n(A_{[k]}(f))^{\frac{1}{n}}.$$

Therefore,
$$\lambda_n(A_{[k]}(f)) \ge \left( \frac{k}{n} \right)^{n} \frac{1}{\|f\|_{\infty}}\|f\|_1.$$
The lemma is proved.
\end{proof}

\textbf{\textit{Proof of Theorem \ref{prop2}.}}
Proposition \ref{cor2} states that
$$\frac{1}{e^k} W_{n-k}(A_{[k]}(f)) \le \frac{1}{\|f\|_{\infty}} W_{n-k}(f) \le W_{n-k} (A_{[k]}(f)).$$

The upper bound from Lemma \ref{lemma1} gives
$$W_{n-k}(A_{[k]}(f)) = W_{n-k}\left(
\frac{A_{[k]}(f)}{\lambda_n(A_{[k]}(f))^{\frac{1}{n}}}\right) \lambda_n(A_{[k]}(f))^{\frac{k}{n}} \le
W_{n-k}\left(
\frac{A_{[k]}(f)}{\lambda_n(A_{[k]}(f))^{\frac{1}{n}}}\right) \frac{e^{\frac{k^2}{n}}}{\|f\|_{\infty}^{\frac{k}{n}}} \|f\|_{1}^{\frac{k}{n}}.$$

The lower bound from Lemma \ref{lemma1} gives
$$W_{n-k}(A_{[k]}(f)) = W_{n-k}\left(
\frac{A_{[k]}(f)}{\lambda_n (A_{[k]}(f))^{\frac{1}{n}}}\right) \lambda_n(A_{[k]}(f))^{\frac{k}{n}} \ge
W_{n-k}\left(
\frac{A_{[k]}(f)}{\lambda_n(A_{[k]}(f))^{\frac{1}{n}}}\right) \left( \frac{k}{n} \right)^k \frac{1}{\|f\|_{\infty}^{\frac{k}{n}}} \|f\|_{1}^{\frac{k}{n}},$$
and the theorem is proved.
\qed

\begin{remark}
{\rm In Theorem \ref{prop2}, the following two-sided estimate was obtained
$$
\frac{k}{e n} \le \left( \frac{W_{n-k}(f)}{\|f\|_1^{\frac{k}{n}} \|f\|_{\infty}^{\frac{n-k}{n}} W_{n-k} \left(
\frac{A_{[k]}(f)}{\lambda_n(A_{[k]}(f))^{\frac{1}{n}}}\right)}\right)^{\frac{1}{k}} \le e.$$
Note that the lower bound is optimal up to multiplication by an absolute constant. To show this we consider the function $f(x) = e^{-|x|}$. Let's calculate the necessary values:
\begin{align*}
&\|f\|_{\infty} = 1, \\
&\|f\|_1 = \int_{\mathbb{R}^n} e^{-|x|} \, dx = \int_{S^{n-1}} \int_0^{+\infty} e^{-r} r^{n-1} \, dr \, d\theta = n \omega_n (n-1)! = \omega_n n! , \\
& \lambda_n(A_{e^{-u}}(f)) = \lambda_n(\{ e^{-|x|} \ge e^{-u} \}) = \lambda_n(\{ |x| \le u \}) = \omega_n u^n, \\
& W_{n-k} (A_{e^{-u}}(f)) = \frac{\omega_n}{\omega_k} \int_{Gr_{n,k}} \lambda_k( P_H (A_{e^{-u}}(f)) ) \, \nu_{n,k}(dH) = \omega_n u^k = g_k(f)^k(u), \\
& W_{n-k}(f) = \int_0^{+\infty} g_k(f)^k(u) e^{-u} \, du = \omega_n \int_0^{+\infty} u^k e^{-u} \, du = \omega_n k!, \\
& W_{n-k}\left( \frac{A_{e^{-k}}(f)}{\lambda_n(A_{e^{-k}}(f))^{\frac{1}{n}}} \right) = \frac{W_{n-k}(A_{e^{-k}}(f))}{\lambda_n(A_{e^{-k}}(f))^{\frac{k}{n}}} = \frac{\omega_n k^k}{(\omega_n k^n)^{\frac{k}{n}}} = \frac{\omega_n}{\omega_n^{\frac{k}{n}}},
\end{align*}
where $\omega_n = \frac{\pi^{\frac{n}{2}}}{\Gamma(\frac{n}{2} + 1)}$ is the volume of the $n$-dimensional unit ball. 
Substituting this, we have
$$
\left( \frac{W_{n-k}(f)}{\|f\|_1^{\frac{k}{n}} \|f\|_{\infty}^{\frac{n-k}{n}} W_{n-k} \left(
\frac{A_{[k]}(f)}{\lambda_n(A_{[k]}(f))^{\frac{1}{n}}}\right)}\right)^{\frac{1}{k}} = \left( \frac{\omega_n k!}{(\omega_n n!)^{\frac{k}{n}} \frac{\omega_n}{\omega_n^{\frac{k}{n}}}} \right)^{\frac{1}{k}} = \frac{k!^{\frac{1}{k}}}{n!^{\frac{1}{n}}} \le e \frac{k}{n}.
$$
}
\end{remark}

\section{The logarithmic growth of the constant in Shepard's problem for functions with respect to the dimension}

This section is devoted to the proof of Theorem \ref{th0}.

We will need the following Minkowski integral inequality for negative powers, which is a simple consequence of the reverse Hölder's inequality (see \cite{HLP}).

\begin{theorem}
\label{th5}
Let $(X, \mu)$, $(Y, \nu)$ be measure spaces (with finite or $\sigma$-finite measures). 
Suppose that $p < 0$ and $f > 0$ is a measurable function on $X \times Y$, such that $\forall x \in X \colon f(x, \cdot) \in L^1(Y, \nu)$ and $\forall y \in Y \colon f(\cdot, y) \in L^p(X, \mu)$, then the following inequality holds
$$
\left\|\int_Y f(\cdot, y) \, \nu(dy) \right\|_{L^p(X, \mu)} \ge \int_Y \| f(\cdot, y) \|_{L^p(X, \mu)} \, \nu(dy).
$$
\end{theorem}

The following statement was proved in the paper \cite{MY} (see Theorem 1.3 there).

\begin{theorem}
\label{th2}
Let $K$ be a convex body in $\mathbb{R}^n$, then the following inequality holds
$$
\left( \int_{Gr_{n,k}} \lambda_k(P_H K)^{-n} \nu_{n,k}(dH) \right)^{-\frac{1}{n}} \ge \frac{\omega_k}{\omega_n^{\frac{k}{n}}} \lambda_n(K)^{\frac{k}{n}}.
$$
\end{theorem}

We now generalize it to the case of log-concave functions. Before that, we formulate the following auxiliary statement. 
By the reasoning similar to the proof of Theorem $4.1$ in \cite{Zh} (see formula ($4.6$) there), we can obtain the following lemma.

\begin{lemma}
\label{lemma6}
Let $f \in Q^n$ and $p \geq 1$, then the inequality
$$
\|f\|_p \le \int_{0}^{+\infty} \lambda_n(A_t(f))^{\frac{1}{p}} \, dt
$$
holds.
\end{lemma}

\begin{theorem}
\label{th3}
For every $f \in Q^n$ the inequality 
$$
\left( \int_{Gr_{n,k}} \|P_H f\|_1^{-n} \, \nu_{n,k}(dH) \right)^{-\frac{1}{n}} \ge \frac{\omega_k}{\omega_n^{\frac{k}{n}}} \|f\|_{\frac{n}{k}}.
$$    
holds.
\end{theorem}
\begin{proof}
By the Fubini's theorem we have
$$
\|P_H f\|_1 = \int_H P_H f(y) \, dy = \int_0^{+\infty} \lambda_k(A_t(P_H f)) \, dt.
$$

Putting the obtained expression into the formula above and using the Minkowski's inequality for negative powers from Theorem \ref{th5}, we get
\begin{multline*}
\left(
\int_{Gr_{n,k}} \|P_H f\|_1^{-n} \, \nu_{n,k}(dH)
\right)^{-\frac{1}{n}} =
\left(
\int_{Gr_{n,k}}
\left(
\int_0^{+\infty} \lambda_k(A_t(P_H f)) \, dt
\right)^{-n} \, \nu_{n,k}(dH)
\right)^{-\frac{1}{n}} \ge \\
\ge
\int_0^{+\infty}
\left(
\int_{Gr_{n,k}} \lambda_k (A_t(P_H f))^{-n} \, \nu_{n,k}(dH)
\right)^{-\frac{1}{n}} \, dt.
\end{multline*}

We note that $A_t(P_H f) = P_H A_t(f)$. By Theorem \ref{th2} for the set $A_t(f)$, we obtain
$$
\left(
\int_{Gr_{n,k}} \|P_H f\|_1^{-n} \, \nu_{n,k}(dH)
\right)^{-\frac{1}{n}} \ge
\int_0^{+\infty} \frac{\omega_k}{\omega_n^{\frac{k}{n}}} \lambda_n(A_t(f))^{\frac{k}{n}} \, dt.
$$

Applying Lemma \ref{lemma6} with $p = \frac{n}{k}$, we obtain the desired inequality.
\end{proof}

For $f \in Q^n$ we consider the following quantity (the affine functional Quermassintegral of the function $f$):
$${\Phi}_{n-k} (f) = \frac{\omega_n}{\omega_k} \Bigl(\int_{Gr_{n,k}} 
\|P_H f\|_{L^1(H)}^n \, \nu_{n, k}(dH) \Bigr)^{\frac{1}{n}}.$$

\begin{lemma}[Affine invariance of ${\Phi}_{n-k}$]
\label{lemma4}
Let $f \in Q^n$ and $T \in SL(n)$.
Then for $k \in \{1, ..., n\}$, the equality ${\Phi}_{n-k} (f) = {\Phi}_{n-k} (Tf)$ holds.
\end{lemma}
\begin{proof}
If $k = n$, then ${\Phi}_0 (f) = \|f\|_1$.

If $k < n$, the proof is similar to the proof of the corresponding theorem for convex bodies (see \cite[Theorem 2]{Grinb}).
\end{proof}

\textbf{\textit{Proof of the Theorem \ref{th0}.}}
\rm{By Theorem \ref{th3} for the function $f_1$, we have
$$
\|f_1\|_{\frac{n}{k}} \le \frac{\omega_n^{\frac{k}{n}}}{\omega_k}
\left( \int_{Gr_{n,k}} \|P_H f_1\|_1^{-n} \, \nu_{n,k}(dH) \right)^{-\frac{1}{n}}.
$$

From the Shepard's condition, we have the inequality
$$
\left(
\int_{Gr_{n,k}} \|P_H f_1\|_1^{-n} \, \nu_{n,k}(dH)
\right)^{-\frac{1}{n}} \le
\left(
\int_{Gr_{n,k}} \|P_H f_2\|_1^{-n} \, \nu_{n,k}(dH)
\right)^{-\frac{1}{n}}.
$$

Lemma \ref{lemma4} implies that for every $T \in SL(n)$ one has
$$
\left(
\int_{Gr_{n,k}} \|P_H f_2\|_1^{-n} \, \nu_{n,k}(dH)
\right)^{-\frac{1}{n}} =
\left(
\int_{Gr_{n,k}} \|P_H T f_2\|_1^{-n} \, \nu_{n,k}(dH)
\right)^{-\frac{1}{n}}.
$$

From Jensen's inequality
$$
\left(
\int_{Gr_{n,k}} \|P_H T f_2\|_1^{-n} \, \nu_{n,k}(dH)
\right)^{-\frac{1}{n}} \le
\int_{Gr_{n,k}} \|P_H T f_2\|_1 \, \nu_{n,k}(dH) = \frac{\omega_k}{\omega_n} W_{n-k}(T f_2).
$$

Using Theorem \ref{prop2} with the function $f = T f_2$, we obtain the inequality
$$
W_{n-k}(T f_2) \le e^{\frac{k^2}{n}}
\|f_2\|_1^{\frac{k}{n}}
\|f_2\|_{\infty}^{\frac{n-k}{n}} W_{n-k} \left(
\frac{A_{[k]}(T f_2)}{\lambda_n(A_{[k]}(T f_2))^{\frac{1}{n}}}\right).
$$

To simplify the expression in brackets we note that
$$
\lambda_n(A_{[k]}(T f_2)) = \lambda_n(T A_{[k]}(f_2)) = \lambda_n(A_{[k]}(f_2)),
$$
thus, the above inequality can be rewritten as
$$
W_{n-k}(T f_2) \le e^{\frac{k^2}{n}}
\|f_2\|_1^{\frac{k}{n}}
\|f_2\|_{\infty}^{\frac{n-k}{n}} W_{n-k} \left(T
\frac{A_{[k]}(f_2)}{\lambda_n(A_{[k]}(f_2))^{\frac{1}{n}}}\right).
$$

Applying the classical Alexandrov inequality \eqref{alex} to the body $T
\frac{A_{[k]}(f_2)}{\lambda_n(A_{[k]}(f_2))^{\frac{1}{n}}}$, we obtain the bound
$$
W_{n-k} \left(T
\frac{A_{[k]}(f_2)}{\lambda_n(A_{[k]}(f_2))^{\frac{1}{n}}}\right) \le \frac{1}{\omega_n^{k-1}} W_{n-1}^k \left(T
\frac{A_{[k]}(f_2)}{\lambda_n(A_{[k]}(f_2))^{\frac{1}{n}}}\right).
$$

Recall that $T$ is an arbitrary operator from $SL(n)$. Varying $T$, we can obtain any affine image of
volume $1$ of the body $\frac{A_{[k]}(f_2)}{\lambda_n(A_{[k]}(f_2))^{\frac{1}{n}}}$. Therefore, from Theorem \ref{prop1}, we get
$$
\min_{T \in SL(n)} W_{n-1} \left(T
\frac{A_{[k]}(f_2)}{\lambda_n(A_{[k]}(f_2))^{\frac{1}{n}}}\right) \le c \omega_n \sqrt{n} \ln(n).
$$

Combining all the previous inequalities and taking into account the inequality $\omega_n^{\frac{1}{n}} \le \frac{C}{\sqrt{n}}$, where $C > 0$ is an absolute constant, we obtain
$$
\|f_1\|_{\frac{n}{k}} \le \frac{\omega_n^{\frac{k}{n}}}{\omega_k} \frac{\omega_k}{\omega_n} e^{\frac{k^2}{n}}
\|f_2\|_1^{\frac{k}{n}} \|f_2\|_{\infty}^{\frac{n-k}{n}} \frac{1}{\omega_n^{k-1}} (c \omega_n \sqrt{n} \ln(n))^k \le
(c_1 \ln(n))^k \|f_2\|_1^{\frac{k}{n}} \|f_2\|_{\infty}^{\frac{n-k}{n}}.
$$
The theorem is proved.}
\qed

\section{Urysohn's reverse inequality for functions}

This section is devoted to the proof of the following theorem generalizing Theorem \ref{th6}.

\begin{theorem}
\label{th1}
Let $k \in \{1, \ldots, n\}$ and $f \in Q^n$, then
$$
\min_{T \in GL(n),\, \|T f\|_1 = 1}W_{n-k} (T f) \le e^{\frac{k^2}{n}} \|f\|_{\infty}^{\frac{n-k}{n}} \min_{T \in GL(n),\, \lambda_n(T A_{[k]}(f)) = 1}W_{n-k} (T A_{[k]}(f)),
$$
\end{theorem}

\begin{proof}
Let $T_0 \in GL(n)$ be such that $\lambda_n(T_0 A_{[k]}(f)) = 1$ and
$$
W_{n-k} (T_0 A_{[k]}(f)) = \min_{\widetilde T \in GL(n),\, \lambda_n(\widetilde T A_{[k]}(f)) = 1}W_{n-k} (\widetilde T A_{[k]}(f)).
$$
Let $\widehat{T}_0 = \alpha T_0$ be such that $\|\widehat T_0 f\|_1 = 1$.
By Theorem \ref{prop2} (with the function $\widehat T_0 f$) we get the inequality
\begin{equation*}
W_{n-k}(\widehat T_0 f) \le e^{\frac{k^2}{n}} \|\widehat T_0 f\|_1^{\frac{k}{n}} \|f\|_{\infty}^{\frac{n-k}{n}} W_{n-k} \left(
\frac{\widehat T_0 A_{[k]}(f)}{\lambda_n(\widehat T_0 A_{[k]}(f))^{\frac{1}{n}}}\right),
\end{equation*}
where we have applied the equality $TA_t(f) = A_t(Tf)$ with $T \in GL(n)$.

Since $\|\widehat T_0 f\|_1 = 1$ and $\lambda_n(\widehat T_0 A_{[k]}(f))^{\frac{1}{n}} = \alpha \lambda_n(T_0 A_{[k]}(f))^{\frac{1}{n}} = \alpha$, then
$$
W_{n-k}(\widehat T_0 f) \le e^{\frac{k^2}{n}} \|f\|_{\infty}^{\frac{n-k}{n}} W_{n-k} (T_0 A_{[k]}(f)).
$$
Taking into account that
$$
\min_{\widetilde T \in GL(n),\, \|\widetilde T f\|_1 = 1}W_{n-k} (\widetilde T f) \le W_{n-k}(\widehat T_0 f),
$$
we obtain the announced statement.
\end{proof}

\begin{corollary}
Let $k \in \{1, \ldots, n\}$ and $f \in Q^n$, then
\begin{multline*}
\max \Big\{\min_{T \in GL(n),\, \|Tf\|_1 = 1}W_{n-k} (Tf) \colon f \in Q^n, \|f\|_{\infty} \le M\Big\}
\le \\
\le e^{\frac{k^2}{n}} M^{\frac{n-k}{n}} 
\max \Big\{\min_{T \in GL(n),\, \lambda_n(TA) = 1} W_{n-k}(TA) \colon A\text{ --- convex body} \Big\}.
\end{multline*}
\end{corollary}

\begin{proof}
In the inequality from Theorem \ref{th1}, we first take the maximum over convex bodies $A$, and then over functions $f \in Q^n$.
\end{proof}

Applying the classical Alexandrov's inequality \eqref{alex} to the bodies in Theorem \ref{th1} and using Theorem \ref{prop1}, we obtain the following generalization of Urysohn's reverse inequality.

\begin{corollary}
Let $k \in \{1, \ldots, n\}$ and $f \in Q^n$, then
$$
\min_{T \in GL(n),\, \|T f\|_1 = 1} W_{n-k} (T f) \le e^{\frac{k^2}{n}} \|f\|_{\infty}^{\frac{n-k}{n}} \omega_n (C \sqrt{n} \ln(n))^k,
$$
where $C$ is an absolute constant.
\end{corollary}


\begin{thebibliography}{99}

\bibitem{A-GMJR}
Alonso--Gutierrez, D., Merino, B. G., Jim\'enez, C.H., Villa, R.:
John’s Ellipsoid and the Integral Ratio of a Log-Concave Function. 
J. Geom. Anal. 28(2), 1182--1201 (2018).
https://doi.org/10.1007/s12220-017-9858-4

\bibitem{A-G}
Alonso--Guti\'errez, D.: A Reverse Rogers--Shephard Inequality for Log-Concave Functions. 
J. Geom. Anal. 29(1), 299--315 (2019).
https://doi.org/10.1007/s12220-018-9991-8

\bibitem{A-AKM}
Artstein--Avidan, S., Klartag, B., Milman, V.:
The Santal\'o point of a function, and a functional form of the Santal\'o inequality.
Mathematika 51(1--2), 33--48 (2004).
https://doi.org/10.1112/S0025579300015497

\bibitem{A-AGM}
Artstein--Avidan, S., Giannopoulos, A., Milman, V.:
Asymptotic Geometric Analysis, Part I.
Mathematical surveys and monographs, vol. 202. American Mathematical Society, pp. xx+451 (2015).
https://doi.org/10.1090/surv/202

\bibitem{BCF}
Bobkov, S.G., Colesanti, A., Fragala, I.:
Quermassintegrals of quasi-concave functions and generalized Prékopa-–Leindler inequalities.
manuscr. math.. 143(1), 131--169 (2014).
https://doi.org/10.1007/s00229-013-0619-9

\bibitem{Bog}
Bogachev, V.I.: Differentiable measures and the Malliavin calculus, Math. Surveys
Monogr., 164, Amer. Math. Soc., Providence, RI (2010).

\bibitem{Bor}
Borell C.: Convex measures on locally convex spaces. Ark. Math, 12(1), 239--252 (1974).

\bibitem{BGVV}
Brazitikos, S., Giannopoulos, A., Valettas, P., Vritsiou, B.H.: 
Geometry of isotropic convex bodies. Am. Math. Soc. (2014).

\bibitem{GK}
Giannopoulos, A., and Koldobsky, A.: 
Variants of the Busemann--Petty Problem and of the Shephard Problem.
Int. Math. Res. Not., 2017(3), 921--943 (2017).
https://doi.org/10.1093/imrn/rnw046

\bibitem{GoK}
Gorev, V., Kosov, E.:
Functional analogs of the Shephard, Busemann-Petty, and Milman problems. To appear in J. Anal. Math. (2024).
https://doi.org/10.48550/arXiv.2208.13308

\bibitem{Grinb}
Grinberg, E.L.:
Isoperimetric inequalities and identities for k-dimensional cross-sections of convex bodies.
Math. Ann. 291, 75--86 (1991).
https://doi.org/10.1007/BF01445191

\bibitem{HLP}
Hardy, G.H., Littlewood, J.E., P\'olya, G:
Inequalities.
Cambridge (1934).
https://doi.org/10.1017/S0025557200143451

\bibitem{KM}
Klartag, B., Milman, V.D.:
Geometry of log-concave functions and measures.
Geom. Dedicata 112(1), 169--182 (2005).
https://doi.org/10.1007/s10711-004-2462-3

\bibitem{Kold}
Koldobsky, A.:
Fourier analysis in convex geometry.
Am. Math. Soc. – №. 116 (2005).
https://doi.org/10.1090/surv/116

\bibitem{MY}
Milman, E., Yehudayoff, A.:
Sharp Isoperimetric Inequalities for Affine Quermassintegrals.
J. Am. Math. Soc, 36(4), 1061-1101. (2023).
https://doi.org/10.1090/jams/1012

\bibitem{Sch}
Schneider, R.:
Convex bodies: the Brunn-–Minkowski theory. – Cambridge university press – №. 151. (2014).
https://doi.org/10.1017/CBO9781139003858

\bibitem{Tz}
Tziotziou, N.:
Inequalities for Sections and Projections of Log-Concave Functions: . J. Geom. Anal., 36(4), 145. (2026).

\bibitem{Zh}
Zhang, G.:
The Affine Sobolev Inequality.
J. Differ. Geom, 53(1), 183-202 (1999).
https://doi.org/10.4310/jdg/1214425451

\end{thebibliography}
\end{document}